\documentclass{amsart}
\usepackage{amssymb,amstext,amsmath,amscd,amsthm,amsfonts,enumerate,graphicx,latexsym,color,stmaryrd,geometry,multicol,ascmac}
\usepackage[all]{xy}
\usepackage{comment}
\newtheorem{thm}{Theorem}[section]
\newtheorem*{thm*}{Theorem}

\newtheorem{cor}[thm]{Corollary}

\newtheorem{prop}[thm]{Proposition}
\theoremstyle{definition}
\newtheorem{conv}[thm]{Convention}
\newtheorem{dfn}[thm]{Definition}
\newtheorem*{dfn*}{Definition}
\newtheorem{rem}[thm]{Remark}
\newtheorem{ques}[thm]{Question}

\newtheorem*{conj*}{Conjecture}
\newtheorem{ex}[thm]{Example}
\newtheorem*{ex*}{Example 4.11}
\newtheorem{nota}[thm]{Notation}
\theoremstyle{remark}
\newtheorem*{ac}{Acknowledgments}

\newtheorem*{claim*}{Claim}
\renewcommand{\qedsymbol}{$\blacksquare$}
\numberwithin{equation}{thm}
\def\a{\mathfrak{a}}

\def\AA{\mathbf{A}}

\def\ann{\operatorname{ann}}
\def\Ass{\operatorname{\mathrm{Ass}}}
\def\b{\mathfrak{b}}

\def\BB{\mathbf{B}}

\def\CC{\mathbf{C}}
\def\cd{\operatorname{cd}}
\def\cl{\operatorname{cl}}

\def\D{\mathsf{D}}
\def\d{\mathrm{D}}

\def\E{\mathbb{E}}
\def\e{\mathrm{E}}
\def\Ext{\mathrm{Ext}}

\def\f{\mathtt{fg}}

\def\ge{\geqslant}
\def\grade{\operatorname{grade}}
\def\H{\mathrm{H}}
\def\height{\operatorname{ht}}
\def\Hom{\mathsf{Hom}}

\def\Im{\operatorname{Im}}
\def\Inj{\operatorname{\mathsf{Inj}}}
\def\Ker{\operatorname{Ker}}
\def\L{\mathcal{L}}
\def\le{\leqslant}
\def\ltensor{\otimes^{\bf{L}}}

\def\m{\mathfrak{m}}
\def\Max{\operatorname{Max}}
\def\mod{\operatorname{\mathsf{mod}}}
\def\Mod{\operatorname{\mathsf{Mod}}}
\def\NN{\mathbb{N}}
\def\P{\mathfrak{P}}
\def\p{\mathfrak{p}}
\def\pd{\operatorname{pd}}
\def\q{\mathfrak{q}}
\def\R{\mathbf{R}}

\def\Spec{\operatorname{Spec}}

\def\supp{\operatorname{supp}}
\def\Supp{\operatorname{Supp}}

\def\X{\mathcal{X}}
\def\xx{\boldsymbol{x}}

\def\V{\mathrm{V}}
\def\Z{\mathbb{Z}}
\begin{document}
\title{Cohomological dimensions of specialization-closed subsets and subcategories of modules}
\author{Hiroki Matsui}
\address[H. Matsui]{Graduate School of Mathematical Sciences\\ University of Tokyo, 3-8-1 Komaba, Meguro-ku, Tokyo 153-8914, Japan}
\email{mhiroki@ms.u-tokyo.ac.jp}
\author{Tran Tuan Nam}
\address[T. T. Nam]{Department of Mathematics-Informatics, Ho Chi Minh University of  Pedagogy, 280 An Duong Vuong, District 5, Ho Chi Minh city, Vietnam}
\email{namtt@hcmue.edu.vn, namtuantran@gmail.com}
\author{Ryo Takahashi}
\address[R. Takahashi]{Graduate School of Mathematics, Nagoya University, Furocho, Chikusaku, Nagoya 464-8602, Japan}
\email{takahashi@math.nagoya-u.ac.jp}
\urladdr{http://www.math.nagoya-u.ac.jp/~takahashi/}
\author{Nguyen Minh Tri}
\address[N. M. Tri]{Department of Natural Science Education, Dong Nai University, Dong Nai, Vietnam}
\email{triminhng@gmail.com}
\author{Do Ngoc Yen}
\address[D. N. Yen]{Faculty of Mathematics \& Computer Science, University of Science, VNU-HCM, Ho Chi Minh City, Vietnam / Faculty of Fundamental Science 2, Posts and Telecommunications Institute of Technology, Ho Chi Minh City, Vietnam}
\email{dongocyen.dhsp@gmail.com}
\thanks{Matsui was partly supported by JSPS Grant-in-Aid for JSPS Fellows 19J00158. Nam, Tri and Yen were funded by Vietnam National Foundation for Science and Technology Development (NAFOSTED) under Grant No. 101.04-2018.304. Takahashi was partly supported by JSPS Grants-in-Aid for Scientific Research 16K05098 and JSPS Fund for the Promotion of Joint International Research 16KK0099.}
\begin{abstract}
Let $R$ be a commutative noetherian ring.
In this paper, we study specialization-closed subsets of $\Spec R$.
More precisely, we first characterize the specialization-closed subsets in terms of various closure properties of subcategories of modules.
Then, for each nonnegative integer $n$ we introduce the notion of $n$-wide subcategories of $R$-modules to consider the question asking when a given specialization-closed subset has cohomological dimension at most $n$.
\end{abstract}
\maketitle
\section{Introduction}

Local cohomology has been introduced by Grothendieck and has been a fundamental tool in commutative algebra and algebraic geometry.
The most important problem concerning local cohomology is to clarify when it vanishes.
In the late 1960s, to explore this problem, Hartshorne \cite{H} has defined the {\em cohomological dimension} $\cd I$ of an ideal $I$ of a commutative ring $R$ as the highest index of the non-vanishing local cohomologies supported on $I$.
This numerical invariant has been studied widely and deeply so far.
Among other things, the celebrated Hartshorne--Lichtenbaum vanishing theorem \cite{H} gives an equivalent condition for $I$ to have $\cd I\le\dim R-1$.
Ogus \cite{O}, Peskine and Szpiro \cite{PS}, and Huneke and Lyubeznik \cite{HL} give characterizations of the ideals $I$ with $\cd I\le\dim R-2$.
Recently, Varbaro \cite{V} and Dao and Takagi \cite{DT} have studied the ideals $I$ with $\cd I\le\dim R-3$.
The cohomological dimension $\cd I$ of an ideal $I$ is naturally extended to the cohomological dimesion $\cd\Phi$ of a specialization-closed subset $\Phi$ of $\Spec R$, that is, the cohomological dimension of an ideal coincides with the cohomological dimension of the Zariski-closed subset defined by the ideal.

A {\em Serre subcategory} of an abelian category is by definition a full subcategory closed under extensions, subobjects and quotient objects.
A {\em localizing subcategory} is defined to be a Serre subcategory closed under coproducts.
These notions were first studied deeply by Gabriel \cite{Gab}.
It was proved that for a commutative noetherian ring $R$ a subset $\Phi$ of $\Spec R$ is specialization-closed if and only if $\Supp^{-1}\Phi$ is localizing, if and only if $\Supp_\f^{-1}\Phi$ is Serre (see Notation \ref{101}).
Since then, localizing subcategories have been investigated by many authors to develop geometric studies of abelian categories; see \cite{Gar,GP1,GP2,Her,K,hcls}.
A {\em wide subcategory} of an abelian category is defined as a full subcategory closed under extensions, kernels and cokernels.
In recent years, wide subcategories have actively been investigated in representation theory of algebras; see \cite{AP,BM,MS,wide,Y}.

In this paper, we first characterize the specialization-closed subsets in terms of various closure properties of subcategories of modules, which complements the above mentioned theorem due to Gabriel; see Theorem \ref{spcl}.
Next, we introduce the notion of an {\em $n$-wide subcategory} for each nonnegative integer $n$.
An $n$-wide subcategory turns out to be nothing but a wide (resp. localizing) subcategory for $n=1$ (resp. $n=0$).
We explore the cohomological dimension of a specialization-closed subset by relating it to the $n$-wide property of a certain corresponding full subcategory of modules.

\begin{thm}[Theorems \ref{nser}, \ref{01}, \ref{100} and \ref{10}]\label{11}
Let $R$ be a commutative noetherian ring.
Let $n\ge0$ be an integer.
\begin{enumerate}[\rm(1)]
\item
Let $\Phi$ be a specialization-closed subset of $\Spec R$.
Then the implication
$$
\cd\Phi \le n\implies\supp^{-1}(\Phi^\complement)\text{ is $n$-wide}
$$
holds true.
The converse holds true as well in each of the following cases.
\begin{enumerate}[\rm(a)]
\item
$n\le1$ or $n\ge\dim R-1$.
\item
$(R,\m,k)$ is complete regular local with $k$ separably closed, $\Phi$ is closed with $\Phi\setminus\{\m\}$ connected, and $n=\dim R-2$.
\item
$R$ has positive prime characteristic, and $\Phi$ is closed with a perfect defining ideal.
\item
$R = S[\Delta]$ is a semigroup ring, and $\Phi$ is closed with a perfect defining ideal generated by elements of $\Delta$.
\end{enumerate}
\item
The following are equivalent.
\begin{enumerate}[\rm(a)]
\item
$\cd\Phi\le n$ for all specialization-closed subsets $\Phi$ of $\Spec R$.
\item
$\supp^{-1}(\Phi^\complement)$ is $n$-wide for all specialization-closed subsets $\Phi$ of $\Spec R$.
\item
$\supp^{-1}((\Max R)^\complement)$ is $n$-wide.
\item
$\dim R\le n$.
\end{enumerate}
\end{enumerate}
\end{thm}

Here, for a subset $\Phi$ of $\Spec R$ we denote by $\Phi^\complement$ the complement of $\Phi$ in $\Spec R$, and by $\Max R$ the set of maximal ideals of $R$.
Also, $\supp$ stands for the {\em small support} introduced by Foxby \cite{Fox}.
Theorem \ref{11} includes the recent theorem of Angeleri H\"ugel, Marks, {\v S}{\v{t}}ov{\'{\i}}{\v{c}}ek, Takahashi and Vit\'{o}ria \cite{epi}, which asserts Theorem \ref{11}(1a) for $n=1$.
It turns out that, whenever $2\le n\le\dim R-2$, the converse of the implication displayed in Theorem \ref{11}(1) does not necessarily hold; see Example \ref{102'}.
The proofs of (1a) and (2) of Theorem \ref{11} use balanced big Cohen--Macaulay modules, whose existence has been shown recently by Andr\'e \cite{And}.

The organization of this paper is as follows.
In Section 2, we state several basic properties of supports, small supports and associated primes to interpret them in terms of closure properties of subcategories of modules.
Section 3 gives preliminaries for the next section.
We introduce the torsion, local cohomology and transform functors with respect to a specialization-closed subset and its cohomological dimension, and investigate their fundamental properties.
The key role is played by the corresponding localization sequence in the derived category.
In Section 4, we define an $n$-wide subcategory of modules for each nonnegative integer $n$.
After verifying several basic properties of them, we consider the $n$-wideness of the subcategory of modules corresponding to a specialization-closed subset, and prove Theorem \ref{11}.

\section{A characterization of the specialization-closed subsets}

We begin with our convention.

\begin{conv}
Throughout this paper, we assume that all rings are commutative and noetherian and all subcategories are full.
We set $\NN=\Z_{\ge0}=\{0,1,2,\dots\}$.
Let $R$ be a ring.
We denote by $\Mod R$ the category of $R$-modules, by $\mod R$ the category of finitely generated $R$-modules, and by $\D(\Mod R)$ the (unbounded) derived category of $\Mod R$.
Note that there are inclusions $\mod R\subseteq\Mod R\subseteq\D(\Mod R)$.
We use well-known facts on local cohomology basically tacitly; we refer the reader to \cite{BS}.
We omit subscripts and superscripts as long as there is no danger of confusion.
\end{conv}

We recall the definitions of the support and the small supports of a complex of modules.

\begin{dfn}\cite{Fox}
For $X \in \D(\Mod R)$ we define the {\em support} of $X$ by $\Supp X= \{\p \in \Spec R \mid X_\p \not\cong 0 \}$, and the {\em small support} of $X$ by $\supp X= \{\p \in \Spec R \mid X \ltensor_R \kappa(\p) \not\cong 0\}$, where $\kappa(\p)=R_\p/\p R_\p$.
\end{dfn}

We denote by $\e_R(M)$ the injective hull of an $R$-module $M$.
Let $X$ be a complex of $R$-modules with $\H^{\ll0}(X)=0$.
Then one can take a {\em minimal injective resolution}
$$
\E_R(X)=(0\to\cdots\xrightarrow{\partial^{i-1}}\E_R^i(X)\xrightarrow{\partial^i}\E_R^{i+1}(X)\xrightarrow{\partial^{i+1}}\cdots)
$$
of $X$, that is, a bounded below complex of injective $R$-modules quasi-isomorphic to $X$ such that $\E_R^i(X)$ is the injective hull of $\Ker\partial^i$ for all $i$.
Recall that the $i$th {\it Bass number} $\mu_i(\p, M)$ of an $R$-module $M$ with respect to a prime ideal $\p$ of $R$ is defined as the dimension of $\Ext_{R_\p}^i(\kappa(\p), M_\p)$ as a $\kappa(\p)$-vector space.
Then $\mu_i(\p, M)$ is equal to the cardinality of the number of direct summands $\e_R(R/\p)$ of $\E_R^i(M)$; see \cite[Theorem 18.7]{M}.
Thus, there is a direct sum decomposition into indecomposable injective modules $\E^i_R(M) \cong \bigoplus_{\p \in \Spec R} \e_R(R/\p)^{\oplus \mu_i(\p, M)}$.

For a subset $\Phi$ of $\Spec R$, let $\cl(\Phi)$ be the set of prime ideals $\p$ of $R$ such that $\q\subseteq\p$ for some $\q\in\Phi$.
This is called the {\em specialization closure}, since it is the smallest specialization-closed subset of $\Spec R$ containing $\Phi$.
We state fundamental properties of supports, small supports and associated primes.

\begin{prop}\label{sup0}
\begin{enumerate}[\rm(1)]
\item
For any $X \in \D(\Mod R)$ one has $\supp X \subseteq \cl(\supp X) \subseteq \Supp X$.
\item
Let $X\in\D(\Mod R)$.
Let $I$ be a complex of injective $R$-modules quasi-isomorphic to $X$.
Then $\supp X\subseteq\bigcup_{i\in\Z}\Ass I^i$.
The equality holds if $\H^{\ll0}(X)=0$ and $I=\E(X)$.
In particular, for each $R$-module $M$ there is an inclusion $\Ass M\subseteq\supp M$, whose equality holds if $M$ is injective.
\item
For every $M\in\Mod R$ it holds that $\cl(\Ass M)  = \cl(\supp M) = \Supp M$.
If $M\in\mod R$, then $\supp M=\Supp M$.
\item
Let $X \to Y \to Z \to X[1]$ be an exact triangle in $\D(\Mod R)$.
Then one has $\supp Y \subseteq \supp X \cup \supp Z$.
\item
For a family $\{X_\lambda\}_{ \lambda \in \Lambda}$ in $\D(\Mod R)$ one has $\supp (\bigoplus_{\lambda \in \Lambda} X_\lambda) = \bigcup_{\lambda \in \Lambda} \supp X_\lambda$.
\item
Let $0 \to L \to M \to N \to 0$ be an exact sequence in $\Mod R$.
Then there are an equality $\Supp M=\Supp L\cup\Supp N$, and inclusions $\Ass L \subseteq \Ass M$ and $\Ass M \subseteq \Ass L \cup \Ass N$.
\item
For a family $\{M_\lambda\}_{ \lambda \in \Lambda}$ in $\Mod R$ one has $\Supp(\bigoplus_{\lambda \in \Lambda} M_\lambda) = \bigcup_{\lambda \in \Lambda} \Supp M_\lambda$ and $\Ass (\bigoplus_{\lambda \in \Lambda} M_\lambda) = \bigcup_{\lambda \in \Lambda} \Ass M_\lambda$.
\end{enumerate}
\end{prop}

\begin{proof}
The first inclusion in (1) is clear, while the second follows from the fact that $\Supp X$ is specialization-closed and contains $\supp X$.
Assertion (2) follows from \cite[Proposition 2.1 and Remark 2.2]{CI}.
The equality and the first inclusion in (6) are obvious, while the second inclusion is shown in \cite[Theorem 6.3]{M}.
Assertions (4) and (5) are straightforward.

Let us show (3).
Let $M\in\Mod R$.
By (1) and (2) we have $\cl(\Ass M)\subseteq\cl(\supp M)\subseteq\Supp M$.
For each $\p \in \Supp M$ the set $\Ass M_\p$ is nonempty, and we find $\q \in \Ass M$ such that $\q \subseteq \p$; see \cite[Theorem 6.2]{M}.
Hence $\p$ belongs to $\cl(\Ass M)$.
If $M$ is finitely generated, then it is observed from \cite[Corollary A.4.16]{C} that $\supp M=\Supp M$.

Now we show (7).
For each $\mu\in\Lambda$ the inclusion $M_\mu\subseteq\bigoplus_{\lambda\in\Lambda}M_\lambda$ shows $\Ass M_\mu\subseteq\Ass(\bigoplus_{\lambda\in\Lambda}M_\lambda)$, which implies $\bigcup_{\mu\in\Lambda}\Ass M_\mu\subseteq\Ass(\bigoplus_{\lambda\in\Lambda}M_\lambda)$.
Let $\p\in\Ass(\bigoplus_{\lambda\in\Lambda}M_\lambda)$.
Then there is a monomorphism $R/\p\hookrightarrow\bigoplus_{\lambda\in\Lambda}M_\lambda$, which factors through a submodule $\bigoplus_{i=1}^nM_{\lambda_i}$ of $\bigoplus_{\lambda\in\Lambda}M_\lambda$ for some finitely many indices $\lambda_1,\dots,\lambda_n\in\Lambda$.
Hence $\p\in\Ass(\bigoplus_{i=1}^nM_{\lambda_i})\subseteq\bigcup_{i=1}^n\Ass M_{\lambda_i}\subseteq\bigcup_{\lambda\in\Lambda}\Ass M_\lambda$, where the first inclusion follows from applying (6) to the split exact sequence $0\to M_{\lambda_j}\to\bigoplus_{i=j}^nM_{\lambda_i}\to\bigoplus_{i=j+1}^nM_{\lambda_i}\to0$ for $1\le j\le n-1$.
\end{proof}

We often use the following notation throughout the paper.

\begin{nota}\label{101}
For $\Theta\in\{\Supp,\supp,\Ass\}$ we denote by $\Theta^{-1}(\Phi)$ the subcategory of $\Mod R$ consisting of modules $X$ with $\Theta(X)\subseteq\Phi$, and put $\Theta^{-1}_\f(\Phi)=\Theta^{-1}(\Phi)\cap\mod R$.
\end{nota}

The following statements are direct consequences of Proposition \ref{sup0}.

\begin{cor}\label{sup1}
Let $\Phi$ be a subset of $\Spec R$.
\begin{enumerate}[\rm(1)]
\item
$\supp_\f^{-1}(\Phi) = \Supp_\f^{-1}(\Phi)$ is a Serre subcategory of $\mod R$, while $\Supp^{-1}(\Phi)$ is a localizing subcategory of $\Mod R$.
\item
$\supp^{-1}(\Phi)$ is closed under direct sums, direct summands and extensions.
\item
The subcategory $\Ass_\f^{-1}(\Phi)$ of $\mod R$ is closed under extensions and submodules, while the subcategory $\Ass^{-1}(\Phi)$ of $\Mod R$ is closed under direct sums, extensions and submodules.
\item
There are inclusions $\Supp^{-1}(\Phi)\subseteq \supp^{-1}(\Phi) \subseteq\Ass^{-1}(\Phi)$ of subcategories of $\Mod R$.
\end{enumerate}
\end{cor}

Now we can give a characterization of specialization-closed subsets in terms of closure properties of subcategories.

\begin{thm}\label{spcl}
The following are equivalent for any subset $\Phi$ of $\Spec R$.
\setlength\multicolsep{0pt}
\begin{multicols}{2}
\begin{enumerate}[\rm(1)]
\item
$\Phi$ is specialization-closed.
\item
$\Ass_\f^{-1}(\Phi)$ is closed under quotient modules.
\item
$\Ass_\f^{-1}(\Phi)$ is Serre.
\item
$\supp^{-1}(\Phi)$ is closed under submodules.
\item
$\supp^{-1}(\Phi)$ is localizing.
\item
$\Ass^{-1}(\Phi)$ is closed under quotient modules.
\item
$\Ass^{-1}(\Phi)$ is localizing.
\item
$\Ass_\f^{-1}(\Phi) = \Supp^{-1}_\f(\Phi)$.
\item
$\Ass^{-1}(\Phi) = \supp^{-1}(\Phi)$.
\item
$\supp^{-1}(\Phi) =  \Supp^{-1}(\Phi)$.
\item
$\Ass^{-1}(\Phi) = \Supp^{-1}(\Phi)$.
\end{enumerate}
\end{multicols}
\end{thm}

\begin{proof}
The implications (7) $\Leftrightarrow$ (6) $\Leftrightarrow$ (1) $\Leftrightarrow$ (2) $\Leftrightarrow$ (3) $\Leftarrow$ (8) hold by Corollary \ref{sup1}(3), \cite[Corollary 2.7]{Kra08} (see also \cite[page 425]{Gab}) and \cite[Theorem 4.1]{wide}.
If $\Phi$ is specialization-closed, then the equality $\cl(\Ass M) = \Supp M$ for $M \in \mod R$ shows $\Ass_\f^{-1}(\Phi) = \Supp^{-1}_\f(\Phi)$.
This proves the implication (1) $\Rightarrow$ (8).
Using the assertions of Corollary \ref{sup1} for $\Mod R$ and Proposition \ref{sup0}, we can easily check that the implications (1) $\Rightarrow$ (11) $\Rightarrow$ (10) $\Rightarrow$ (5) $\Rightarrow$ (4) and (11) $\Rightarrow$ (9) $\Rightarrow$ (4) hold.

It remains to show the implication (4) $\Rightarrow$ (1).
Assume (4) and take $\p \in \Phi$.
Then $R/\p \subseteq\e_R(R/\p) \in \supp^{-1}(\Phi)$, and hence $R/\p \in \supp^{-1}(\Phi)$.
Therefore $\V(\p)=\Supp R/\p=\supp R/\p\subseteq\Phi$, and thus $\Phi$ is specialization-closed.
\end{proof}

\section{Local cohomology with respect to a specialization-closed subset}

First of all, we introduce the torsion functor and the local cohomology functor with respect to a specialization-closed subset, extending the torsion functor and the local cohomology functor with respect to an ideal.

\begin{dfn}
Let $\Phi$ be a subset of $\Spec R$.	
We define the {\it $\Phi$-torsion functor} $\Gamma_{\Phi}: \Mod R \to \Mod R$ by $\Gamma_\Phi(M)= \{x \in M \mid \Supp(Rx) \subseteq \Phi \}$ for $M\in\Mod R$.
If $\Phi$ is specialization-closed, then we have natural isomorphisms $\Gamma_\Phi(-) \cong \varinjlim_{\V(I)\subseteq\Phi} \Gamma_I(-)\cong\bigcup_{\V(I)\subseteq\Phi}\Gamma_I(-)$.
The first isomorphism says that $\Gamma_\Phi: \Mod R \to \Mod R$ is a left exact functor, and for each integer $n$ we can consider its $n$th right derived functor $\H^n_\Phi: \Mod R \to \Mod R$, which we call the $n$th {\em local cohomology functor} with respect to $\Phi$.
There is a natural isomorphism $\H^n_\Phi(-)\cong \varinjlim_{\V(I)\subseteq\Phi} \H^n_I(-)$.
\end{dfn}

We state several basic properties of the torsion and local cohomology functors with respect to a specialization-closed subset $\Phi$, which are well-known in the case where $\Phi$ is closed.

\begin{prop}\label{1}
Let $\Phi$ be a specialization-closed subset of $\Spec R$.
\begin{enumerate}[\rm(1)]
\item
For an $R$-module $M$, there are implications\\
{\rm(a)} $\Gamma_\Phi(M)=M\Leftrightarrow\Supp M\subseteq\Phi\Leftrightarrow\Ass M\subseteq\Phi\Rightarrow\H_\Phi^{>0}(M)=0$,\qquad
{\rm(b)} $\Gamma_\Phi(M)=0\Leftrightarrow\Ass M\subseteq\Phi^\complement$.
\item
For an integer $n$, one has $\Supp \H_\Phi^n(M) \subseteq \Phi$, $\Gamma_\Phi(\H_\Phi^n(M))=\H_\Phi^n(M)$ and $\H_\Phi^{>0}(\H_\Phi^n(M)) = 0$.
\item
One has $\Gamma_\Phi(M/\Gamma_\Phi(M))=0$ and $\H_\Phi^i(M/\Gamma_\Phi(M))\cong\H_\Phi^i(M)$ for all $i>0$.
\item
For a family $\{M_\lambda\}_{\lambda\in\Lambda}$ in $\Mod R$ and an integer $n$, there is a natural isomorphism $\H_\Phi^n(\bigoplus_{\lambda\in\Lambda}M_\lambda)\cong\bigoplus_{\lambda\in\Lambda}\H_\Phi^n(M_\lambda)$.
\item
Let $\p$ be a prime ideal of $R$.
Let $\Phi$ be a specialization-closed subset of $\Spec R$, and put
$$
\Phi_\p=\{P\in\Spec R_\p\mid P\cap R\in\Phi\}.
$$
Then one has an isomorphism $\H_\Phi^n(M)_\p\cong\H_{\Phi_\p}^n(M_\p)$ for all $R$-modules $M$ and integers $n$.
\end{enumerate}
\end{prop}

\begin{proof}
(1a) Note that $\Supp M=\bigcup_{x\in M}\Supp Rx$.
This implies that $\Gamma_\Phi(M)=M$ if and only if $\Supp M\subseteq\Phi$, if and only if $\Ass M\subseteq\Phi$ since $\Phi$ is specialization-closed.
When $\Supp M\subseteq\Phi$, one has $\Gamma_\Phi(\E^n(M))=\E^n(M)$ for all $n \ge 0$ by Proposition \ref{sup0}(1)(2).
Hence $\Gamma_\Phi(\E(M)) = \E(M)$, which implies $\H_\Phi^{>0}(M)= 0$.
Thus the first assertion follows.

(1b) If $x$ is a nonzero element of $M$ with $\Supp Rx\subseteq\Phi$, then $\emptyset\ne\Ass Rx\subseteq\Supp Rx$ and hence $\Ass Rx\cap\Phi$ is nonempty.
If $\p$ is a prime ideal in $\Ass M\cap\Phi$, then $\p=\ann(x)$ for some $x\in M$.
As $R/\p\cong Rx$, we have $\Ass Rx=\{\p\}\subseteq\Phi$, which implies $\Supp Rx\subseteq\Phi$ since $\Phi$ is specialization-closed.
Now the second assertion follows.

(2) By (1) a prime ideal $\p$ is in $\Phi$ if and only if $\Gamma_\Phi(\e_R(R/\p))\ne0$, if and only if $\Gamma_\Phi(\e_R(R/\p)) = \e_R(R/\p)$.
We see that $\Supp \Gamma_\Phi(\E^n(M)) \subseteq \Phi$ for an $R$-module $M$ and an integer $n\ge0$.
Since $\Supp^{-1}_{\Mod R}(\Phi)$ is a localizing subcategory of $\Mod R$ by Corollary \ref{sup1}(1), we have $\Supp \H_\Phi^n(M) \subseteq \Phi$.
Therefore $\Gamma_\Phi(\H_\Phi^n(M))=\H_\Phi^n(M)$ and $\H_\Phi^{>0}(\H_\Phi^n(M))=0$ by (1a).

(3) The assertion is easy to deduce from assertion (2).

(4) The functor $\Gamma_\Phi=\varinjlim_{\V(I)\subseteq\Phi}\Gamma_I$ commutes with direct sums of modules.
As $R$ is noetherian, $\bigoplus_{\lambda\in\Lambda}\E(M_\lambda)$ gives an injective resolution of $\bigoplus_{\lambda \in \Lambda} M_\lambda$; see \cite[Theorem 18.5]{M} for instance.
We obtain isomorphisms $\H_\Phi^n(\bigoplus_{\lambda\in\Lambda}M_\lambda) \cong \H^n(\Gamma_\Phi(\bigoplus_{\lambda \in \Lambda} \E(M_\lambda)))\cong \bigoplus_{\lambda \in \Lambda} \H^n(\Gamma_\Phi( \E(M_\lambda))) \cong \bigoplus_{\lambda\in\Lambda}\H_\Phi^n(M_\lambda)$.

(5) The assignments $\V(I)\mapsto\V(IR_\p)$ and $\V(J)\mapsto\V(J\cap R)$, where $I,J$ are ideals of $R,R_\p$ respectively, give mutually inverse inclusion-preserving bijections between the set $\AA$ of closed subsets $Z$ of $\Spec R$ with $\p\in Z\subseteq\Phi$ and the set $\BB$ of closed subsets $W$ of $\Spec R_\p$ with $\emptyset\ne W\subseteq\Phi_\p$.
Note that $\p\in Z$ if $\H_Z^n(M)_\p\ne0$, and $\H_W^n(M_\p)=0$ if $W=\emptyset$.
Thus
$$
\H_\Phi^n(M)_\p
\cong(\varinjlim_{Z\in\CC}\H_Z^n(M))_\p
\cong\varinjlim_{Z\in\CC}(\H_Z^n(M)_\p)
=\varinjlim_{Z\in\AA}(\H_Z^n(M)_\p)\\
\cong\varinjlim_{W\in\BB}(\H_W^n(M_\p))
=\varinjlim_{W\in\CC_\p}(\H_W^n(M_\p))
\cong\H_{\Phi_\p}^n(M_\p),
$$
where $\CC$ (resp. $\CC_\p$) stands for the set of closed subsets of $\Spec R$ (resp. $\Spec R_\p$) contained in $\Phi$ (resp. $\Phi_\p$).
\end{proof}

Next we define the cohomological dimension of a specialization-closed subset, which extends the celebrated invariant of the cohomological dimension of an ideal.
\begin{dfn}
Let $\Phi$ be a specialization-closed subset of $\Spec R$.
We define the {\it cohomological dimension} of $\Phi$ by
$$
\cd\Phi
=\sup\{i\in\Z\mid\H_\Phi^i(M)\ne0\text{ for some }M\in\Mod R\}
=\inf\{i\in\Z\mid\H_\Phi^{>i}(M)=0\text{ for all }M\in\Mod R\}
\in\NN\cup\{\infty,-\infty\}.
$$
For an ideal $I$ of $R$ we set $\cd I=\cd\V(I)$ and call it the {\em cohomological dimension} of $I$.
\end{dfn}

\begin{rem}\label{rem-}
Let $\Phi$ be a specialization-closed subset of $\Spec R$.
Then:\\
(1) $\cd\Phi=-\infty \Leftrightarrow \Phi=\emptyset$.\qquad
(2) $\cd\Phi\le 0\Leftrightarrow\Gamma_\Phi$ is exact.\qquad
(3) $\cd\Phi \le \dim R$.\qquad
(4) $\cd\Phi \le \sup\{\cd I\mid\V(I)\subseteq\Phi\}$.\\
Indeed, if $\p\in\Phi$, then $\H_\Phi^0(\e_R(R/\p))\ne0$ by Proposition \ref{1}(1b), which deduces (1).
Item (3) follows from Grothendieck's vanishing theorem \cite[6.1.2]{BS}, while (2) and (4) are clear.
\end{rem}

The following proposition is well-known in the case where $\Phi$ is closed.

\begin{prop}\label{cdr}
Let $\Phi$ be a specialization-closed subset of $\Spec R$.
Assume either that $R$ has finite Krull dimension or that $\Phi$ is closed.
Then there is an equality $\cd\Phi=\sup\{i\in\Z\mid\H_\Phi^i(R)\ne0\}$.
\end{prop}

\begin{proof}
It suffices to prove that $\cd\Phi \le n$ if and only if $\H_\Phi^{>n}(R)=0$ for each $n\in\Z$.
We may assume $\H_\Phi^{>r}(R)=0$ for some integer $r$.
Indeed, if $\dim R=d < \infty$, then $\H_\Phi^{i}(R)=0$ for all $i > d$.
If $\Phi = \V(I)$ for some ideal $I$, then $\H_\Phi^{i}(R)=0$ for all $i > s$, where $I$ is generated by $s$ elements.

The ``only if'' part is evident.
To show the ``if'' part, assume $\H_\Phi^{>n}(R)=0$ and let $M$ be an $R$-module.
Take a free resolution $\cdots\xrightarrow{f_2}F_1\xrightarrow{f_1}F_0\xrightarrow{f_0}M\to0$ and set $M_i=\Im f_i$ for each $i$.
Using Proposition \ref{1}(4), we get $\H_\Phi^i(M) \cong \H_\Phi^{i+1}(M_1) \cong \cdots \cong \H_\Phi^{t}(M_{t-i}) =0$ for any $i > n$, where $t=\max\{i,r+1\}$.
Thus $\cd\Phi\le n$ as desired.
\end{proof}

For a subset $\Phi$ of $\Spec R$, put $\L_\Phi=\{X\in\D(\Mod R)\mid\supp X\subseteq\Phi\}$ and $\L_\Phi^\perp=\{X \in \D(\Mod R) \mid \Hom_{\D(\Mod R)}(\L_\Phi, X) = 0\}$.
Then by \cite[Theorem 2.8]{Nee} the subcategory $\L_\Phi$ is generated by a set.
It follows from \cite[Propositions 4.9.1, 4.11.2 and Theorem 5.6.1]{Kra10} and \cite[Lemma 4.5]{Nee11} that there exist exact functors
$$
\xymatrix{
\L_\Phi\ar@<2pt>[r]^-{i_\Phi} & \D(\Mod R)\ar@<2pt>[r]^-{\lambda_\Phi}\ar@<2pt>[l]^-{\gamma_\Phi} & \L_\Phi^\perp,\ar@<2pt>[l]^-{j_\Phi}
}
$$
where $i_\Phi$ and $j_\Phi$ are inclusion functors with $i_\Phi \dashv \gamma_\Phi$ and $\lambda_\Phi \dashv j_\Phi$, such that:
\begin{enumerate}[\rm(1)]
\item
There is a functorial exact triangle $\gamma_\Phi (X) \xrightarrow{\theta_\Phi(X)} X \xrightarrow{\eta_\Phi(X)} \lambda_\Phi (X) \to \gamma_\Phi (X)[1]$ for $X\in\D(\Mod R)$ which is isomorphic to any exact triangle $X' \to X \to X'' \to X'[1]$ with $X' \in \L_\Phi$ and $X'' \in \L_\Phi^\perp$.

\item
If $\Phi$ is specialization-closed, then $\gamma_\Phi \cong \R\Gamma_\Phi$ and $\L_\Phi^\perp = \L_{\Phi^\complement}$.
\end{enumerate}

For a complex homologically bounded below, the image by $\lambda_\Phi$ can be described by using its minimal injective resolution as follows.

\begin{prop}\label{lmd}
Let $\Phi$ be a specialization-closed subset of $\Spec R$.
Let $X \in \D(\Mod R)$ with $\H^{\ll0}(X)=0$.
Then there is an isomorphism $\lambda_\Phi(X)\cong\E(X)/\Gamma_\Phi(\E(X))$ in $\D(\Mod R)$.
\end{prop}

\begin{proof}
Fix an integer $n$.
Write $\E^n(X)= \bigoplus_{\p \in \Spec R} \e_R(R/\p)^{\oplus A_{n,\p}}$, where $A_{n,\p}$ is a set.
Then the lower left holds, which implies the lower right by Proposition \ref{sup0}(2).
$$
\begin{cases}
\Gamma_\Phi(\E^n(X))\textstyle= \bigoplus_{\p \in \Phi} \e_R(R/\p)^{\oplus A_{n,\p}},\\
\E^n(X)/ \Gamma_\Phi(\E^n(X))\textstyle= \bigoplus_{\p \in \Phi^\complement} \e_R(R/\p)^{\oplus A_{n,\p}}.
\end{cases}
\quad
\begin{cases}
\supp \Gamma_\Phi(\E(X))=\bigcup_{n\in\Z} \Ass \Gamma_\Phi(\E^n(X)) \subseteq \Phi,\\
\supp (\E(X)/\Gamma_\Phi(\E(X)))=\bigcup_{n\in\Z} \Ass (\E^n(X)/\Gamma_\Phi(\E^n(X))) \subseteq \Phi^\complement.
\end{cases}
$$
Thus, the natural exact triangle $\Gamma_\Phi(\E(X)) \to \E(X) \to \E(X)/\Gamma_\Phi(\E(X)) \to \Gamma_\Phi(\E(X))[1]$ is isomorphic to the exact triangle $\gamma_\Phi(X) \xrightarrow{\theta_\Phi(X)} X \xrightarrow{\eta_\Phi(X)} \lambda_\Phi (X) \to \gamma_\Phi (X)[1]$ in $\D(\Mod R)$, which shows the assertion of the proposition.
\end{proof}

Next we introduce the transform functor with respect to a specialization-closed subset, which is also a generalization of the transform functor of an ideal.

\begin{dfn}
Let $\Phi$ be a subset of $\Spec R$.
The {\it $\Phi$-transform functor} $\d_\Phi^n: \Mod R \to \Mod R$ is defined by $\d_\Phi^{n}(M)= \H^n(\lambda_\Phi (M))$ for $M\in\Mod R$.
Applying $\H^0$ to $\eta_\Phi(M): M \to \lambda_\Phi(M)$, we get a natural map $\zeta_\Phi(M): M \to \d_\Phi^0(M)$.
\end{dfn}

\begin{rem}\label{trans}
Let $\Phi$ be a specialization-closed subset of $\Spec R$.
\begin{enumerate}[\rm(1)]
\item
For each $M\in\Mod R$, there are equivalences
$$
M \in \supp_{\Mod R}^{-1}(\Phi^\complement)
\ \Leftrightarrow\ 
\R\Gamma_\Phi(M) \cong 0
\ \Leftrightarrow\ 
\eta_\Phi(M): M \xrightarrow{\cong} \lambda_\Phi(M)
\ \Leftrightarrow\ 
\zeta_\Phi(M): M \xrightarrow{\cong} \d^0_\Phi(M) \mbox{ and } \d_\Phi^{>0}(M) = 0.
$$
In fact, if $\R\Gamma_\Phi(M) \cong 0$, then $\gamma_\Phi(M)=0$ and $M\cong \lambda_{\Phi}(M) \cong \E(M)/\Gamma_\Phi(\E(M))$ by Proposition \ref{lmd}, which implies $M \in \supp_{\Mod R}^{-1}(\Phi^\complement)$.
The other implications follow from Propositions \ref{sup0}(2) and \ref{1}(1).
\item
Let $M \in \Mod R$.
The exact triangle $\R\Gamma_\Phi(M) \xrightarrow{\theta_\Phi(M)} M \xrightarrow{\eta_\Phi(M)} \lambda_\Phi (M) \to \gamma_\Phi(M) [1]$ yields:\\
(i) $\d_\Phi^{<0}(M) =0$,\quad
(ii) $0 \to \H_\Phi^0(M) \to M \xrightarrow{\zeta_\Phi(M)} \d_\Phi^{0}(M) \to \H^1_\Phi(M) \to 0$ is exact,\quad
(iii) $\d_\Phi^{i}(M) \cong \H^{i+1}_\Phi(M)$ for $i \ge 1$.

\noindent
In particular, for an injective $R$-module $I$, there is a natural isomorphism $\d_\Phi^0(I) \cong I/\Gamma_\Phi(I)$.
\item
Let $0 \to L \to M \to N \to 0$ be an exact sequence in $\Mod R$.
Then it induces an exact triangle $\lambda_\Phi(L) \to \lambda_\Phi(M) \to \lambda_\Phi(N) \to \lambda_\Phi(L)[1]$, which induces a long exact sequence $0 \to \d_\Phi^0(L) \to \d_\Phi^0(M) \to \d_\Phi^0(N) \to \d_\Phi^1(L) \to\cdots$.
This means that the functor $\d_\Phi^0$ is left exact and the sequence $(\d_\Phi^i)_{i \ge 0}$ is a {\it cohomological $\delta$-functor} in the sense of \cite[Chapter XX, \S7]{Lan}.
Thus we can consider the right derived functor $\R\d_\Phi^0: \D^+(\Mod R) \to \D^+(\Mod R)$ on the derived category of bounded below complexes, and we have an isomorphism $\d_\Phi^i \cong \H^i\R\d_\Phi^0$ for $i\in\Z$.
Actually, there are natural isomorphisms $\lambda_\Phi(X) \cong \E(X)/\Gamma_\Phi(\E(X)) \cong \d_\Phi^0(\E(X)) \cong \R\d_\Phi(X)$ for $X \in \D^+(\Mod R)$, where the first isomorphism follows from Proposition \ref{lmd}.
For an ideal $I$ of $R$, the $\V(I)$-transform functors are nothing but the $I$-transform functors: $\d_{\V(I)}^n \cong \d_I^n := \varinjlim_{i \in \NN} \Ext_R^n(I^i, -)$ for all $n\in\Z$.
This follows from the above argument and \cite[Exercise 2.2.2]{BS}: $\d_{\V(I)}^n(M) \cong \H^n\R\d_{\V(I)}^0(M) \cong \d_I^n(M)$ for all $n \in \Z$.
\end{enumerate}
\end{rem}

We state the relationship between $\Phi$-transform functors and local cohomology functors with respect to $\Phi$, which gives a generalization of \cite[Corollary 2.2.8]{BS}.

\begin{prop}\label{trs}
Let $\Phi$ be a specialization-closed subset of $\Spec R$.
Let $M$ be an $R$-module.
\begin{enumerate}[\rm(1)]
\item
One has $\d_\Phi^0(\H^0_\Phi(M)) = 0$.
\item
The natural map $\d_\Phi^0(M) \to \d_\Phi^0(M/\H^0_\Phi(M))$ is an isomorphism.
\item
The equality $\d_\Phi^0(\zeta_\Phi(M))=\zeta_\Phi(\d_\Phi^0(M))$ holds, which is an isomorphism $\d_\Phi^0(M) \to \d_\Phi^0(\d_\Phi^0(M))$.
\item
It holds that $\H^0_\Phi(\d_\Phi^0(M)) = 0 = \H^1_\Phi(\d_\Phi^0(M))$.
\item
The map $\H^n_\Phi(\zeta_\Phi(M)) : \H_\Phi^n(M) \to \H_\Phi^n(\d^0_\Phi(M))$ is an isomorphism for $n \ge 2$.
\end{enumerate}
\end{prop}

\begin{proof}
(1) The assertion can be deduced from Proposition \ref{1}(2) and Remark \ref{trans}(2)(ii).

(2) Apply the functor $\d_\Phi^0$ to the short exact sequence $0 \to \H^0_\Phi(M) \to M \to M/\H^0_\Phi(M) \to 0$, and use (1), Remark \ref{trans}(2)(iii) and Proposition \ref{1}(2).

(3) By Proposition \ref{lmd}, for any $M \in \Mod R$ there is a commutative diagram with exact rows:
$$
\xymatrix{
0\ar[r] & M  \ar[r]\ar[d]_{\zeta_\Phi(M)} & \E^0(M)\ar[r]\ar[d] & \E^1(M)\ar[d] \\
0\ar[r] & \d_\Phi^0(M)\ar[r] & \E^0(M)/\Gamma_\Phi(\E^0(M))\ar[r] & \E^1(M)/\Gamma_\Phi(\E^1(M))
}
$$
Let $\X$ be the subcategory of $\Mod R$ consisting of $R$-modules $M$ with $\Ass \E^i(M) \subseteq \Phi^\complement$ for $i=0,1$.
We establish a claim.

\begin{claim*}
The subcategory $\X$ of $\Mod R$ is reflective, that is, the inclusion functor $r:\X \hookrightarrow \Mod R$ admits a left adjoint $l$, which is $\d_{\Phi}^0: \Mod R \to \X$.
Furthermore, the unit $u$ of this adjunction is $\zeta_{\Phi}$.
\end{claim*}

\begin{proof}[Proof of Claim]
Take an $R$-module $N\in\X$.
There is a natural homomorphism $\Theta:\Hom_R(\d_{\Phi}^0(M), N) \to \Hom_R(M, N)$ given by $f \mapsto f \circ \zeta_{\Phi}(M)$.
It suffices to verify that $\Theta$ is an isomorphism.
Take $g \in \Hom_R(M, N)$ and let $g^i\in\Hom_R(\E^i(M),\E^i(N))$ be an extension of $g$ for each $i$.
Fix $i\in\{0,1\}$.
As $N\in\X$, we have $\Gamma_\Phi(\E^i(N))=0$ and $\Gamma_\Phi(g^i)=0$.
Thus $g^i$ factors through $\E^i(M)/\Gamma_\Phi(\E^i(M))$, and hence $g$ factors through $\d_\Phi^0(M)$.
This shows the surjectivity of $\Theta$.
Next, take $f \in \Hom_R(\d_{\Phi}^0(M), N)$ with $f \circ\zeta_\Phi(M) = 0$.
Then the composition $p:\E^0(M) \to \E^0(M)/\Gamma_{\Phi}(\E^0(M)) \xrightarrow{f^0} \E^0(N)$ factors the $0$th differential $d^0:\E^0(M) \to \E^1(M)$, that is, there is a map $s:\E^1(M)\to\E^0(N)$ with $p=sd$.
Similarly as above, $\Gamma_\Phi(s)=0$ and $s$ factors through $\E^1(M)/\Gamma_{\Phi}(\E^1(M))$.
Therefore $f^0$ factors through the $0$th differential of $\E(M)/\Gamma_{\Phi}(\E(M))$.
This shows $f=0$.
\renewcommand{\qedsymbol}{$\square$}
\end{proof}

Let $c$ be the counit of the adjunction.
Then the counit-unit equations are $1_l=cl\circ lu$ and $1_r=rc\circ ur$.
Hence $1_{rl}=rcl\circ rlu$ and $1_{rl}=rcl\circ url$.
As the right adjoint $r$ is the inclusion functor, which is fully faithful.
Hence the counit $c$ is an isomorphism, and so is $rcl$.
Therefore the equality $rlu=url$ holds and it is an isomorphism.
This means that the equality $\zeta_\Phi\cdot\d_\Phi^0=\d_\Phi^0\cdot\zeta_\Phi$ holds and it is an isomorphism.

(4) Thanks to (2), we may assume $\H_\Phi^{0}(M) = 0$.
Applying the snake lemma to the commutative diagram
$$
\xymatrix@C4pc{
0 \ar[r] & \d_\Phi^{0} (M) \ar[r]^-{\d_\Phi^0(\zeta_\Phi(M))} & \d_\Phi^0(\d_\Phi^{0}(M)) \ar[r] & \d_\Phi^{0} (\H^1_\Phi(M)) \ar[r] & \d_\Phi^{1} (M) \\
0 \ar[r] & M \ar[r]^-{\zeta_\Phi(M)} \ar[u]_{\zeta_\Phi(M)} & \d_\Phi^0(M) \ar[r]  \ar[u]_{\zeta_\Phi(\d_\Phi^0(M))} & \H_\Phi^1(M) \ar[r] \ar[u]_{\zeta_\Phi(\H_\Phi^1(M))} & 0
}
$$
induced by (2)(ii) and (3) of Remark \ref{trans}, we have an exact sequence
$$
\H^0_\Phi(M) = 0 \to \H^0_\Phi(\d^0_\Phi(M)) \to \H^0_\Phi(\H^1_\Phi(M))=\H^1_\Phi(M) \xrightarrow{\delta} \H^1_\Phi(M) \to \H^1_\Phi(\d^0_\Phi(M)) \to \H^1_\Phi(\H^1_\Phi(M))=0,
$$
where the second and third equalities follow from Proposition \ref{1}(2).
It is seen from (3) that the map $\delta$ is the identity map.
The above exact sequence now tells us that $\H^0_\Phi(\d^0_\Phi(M)) = 0=\H^1_\Phi(\d^0_\Phi(M))$.

(5) As $\H_\Phi^{>0}(\H_\Phi^0(M))=0$ by Proposition \ref{1}(2), the natural map $\H_\Phi^n(M)\to\H_\Phi^n(M/\H_\Phi^0(M))$ is an isomorphism.
Thus we may assume $\H_\Phi^0(M) = 0$.
Since $\H_\Phi^{>0}(\H_\Phi^1(M)) = 0$ by Proposition \ref{1}(2) again, the long exact sequence induced from $0 \to M \to \d_\Phi^0(M) \to \H_\Phi^1(M) \to 0$ completes the proof.
\end{proof}

\section{$n$-wide subcategories}\label{n}

We start by giving the definition of an $n$-wide subcategory of $\Mod R$.

\begin{dfn}
Let $n\ge0$ be an integer.
A subcategory $\X$ of $\Mod R$ is said to be closed under {\em $n$-kernels} (resp. {\em $n$-cokernels}) if for every exact sequence $0 \to M \to X^0 \to \cdots\to X^n\to N\to0$ in $\Mod R$ with $X^i \in \X$ for all $i$ the module $M$ (resp. $N$) is in $\X$.
We say that $\X$ is {\it $n$-wide} if it is closed under  extensions, $n$-kernels and $n$-cokernels.
\end{dfn}

\begin{rem}\label{rem2}
\begin{enumerate}[(1)]
\item
A subcategory $\X$ of $\Mod R$ is $n$-wide if and only if for an exact sequence
$$
M_n\to\cdots\to M_0\to M\to M^0\to\cdots\to M^n
$$
in $\Mod R$ with $M_i,M^i\in\X$ for all $i$ one has $M\in\X$.
\item
If a subcategory of $\Mod R$ is closed under $n$-kernels (resp. $n$-cokernels), then it is closed under $(n+1)$-kernels (resp. $(n+1)$-cokernels).
In particular, being $n$-wide implies being $(n+1)$-wide.
\item
Let $\X_1,\dots,\X_r$ be subcategories of $\Mod R$.
If $\X_i$ is closed under $n_i$-kernels (resp. $n_i$-cokernels) for all $i$, then $\X_1\cap\cdots\cap\X_r$ is closed under $\max\{n_1,\dots,n_r\}$-kernels (resp. $\max\{n_1,\dots,n_r\}$-cokernels).
In particular, if $\X_i$ is $n_i$-wide for all $i$, then $\X_1\cap\cdots\cap\X_r$ is $\max\{n_1,\dots,n_r\}$-wide.
\item
A subcategory of $\Mod R$ is closed under $0$-kernels (resp. $0$-cokernels) if and only if it is closed under submodules (resp. quotient modules).
In particular, being $0$-wide and closed under direct sums is equivalent to being localizing.
\item
A subcategory of $\Mod R$ is closed under $1$-kernels (resp. $1$-cokernels) if and only if it is closed under kernels (resp. cokernels).
In particular, being $1$-wide is equivalent to being wide.
\end{enumerate}
\end{rem}

We give a necessary condition for a specialization-closed subset to have cohomological dimension at most $n$ in terms of an $n$-wide subcategory.

\begin{thm}\label{nser}
Let $\Phi\subseteq\Spec R$ be specialization-closed, and $n\ge0$ an integer.
If $\cd\Phi \le n$, then $\supp_{\Mod R}^{-1}(\Phi^\complement)$ is $n$-wide.
\end{thm}

\begin{proof}
Consider an exact sequence $0 \to M\xrightarrow{f_{n+1}} X_{n} \xrightarrow{f_{n}}\cdots \xrightarrow{f_1} X_0 \xrightarrow{f_0} N \to 0$ with $X_i \in \supp_{\Mod R}^{-1}(\Phi^\complement)$ for all $i$.
Then $\H_\Phi^{\ge0}(X_i)=0$ by Remark \ref{trans}(1).
Set $U_k= \Im f_k$ for $0\le k \le n+1$.
The exact sequence $0 \to U_{k+1} \to X_k \to U_k \to 0$ tells $\H_\Phi^0(U_{k+1})=0$ and $\H_\Phi^i(U_k)\cong \H_\Phi^{i+1}(U_{k+1})$ for $i \ge 0$ and $0 \le k \le n$.
Thus $\H_{\Phi}^i(N) \cong \H_{\Phi}^{i+1}(U_1) \cong \cdots \cong \H_{\Phi}^{i+n+1}(U_{n+1}) =0$ for $i \ge 0$.
Also, $\H_{\Phi}^i(M) \cong \H_{\Phi}^{i-1}(U_n) \cong \cdots \cong \H_{\Phi}^{0}(U_{n-i+1}) =0$ for $1\le i \le n$ and $\H_\Phi^0(M)=\H_\Phi^0(U_{n+1})=0$.
It follows that $\H_\Phi^{\ge0}(M)=\H_\Phi^{\ge0}(N)=0$, and Remark \ref{trans}(1) implies that $M, N$ belong to $\supp_{\Mod R}^{-1}(\Phi^\complement)$.
\end{proof}

It is natural to ask whether the converse of the implication in Theorem \ref{nser} holds.

\begin{ques}\label{q}
Let $\Phi$ be a specialization-closed subset of $\Spec R$, and let $n\ge0$ be an integer.
Suppose that $\supp_{\Mod R}^{-1}(\Phi^\complement)$ is $n$-wide.
Then, does it hold that $\cd\Phi\le n$?
\end{ques}

We shall prove that this question is affirmative in several cases.
For it, we need some preparation.

\begin{prop}\label{6}
Let $\Phi$ be a specialization-closed subset of $\Spec R$.
Let $M$ be an $R$-module.
Let $n\ge0$ be an integer.
Then the following are equivalent.
\begin{enumerate}[\rm(1)]
\item
One has $\H_\Phi^i(M) = 0$ for all $0 \le i \le n$.
\item
The inclusion $\Ass \E^i(M) \subseteq \Phi^\complement$ holds for all $0 \le i \le n$.
\item
There exists an injective resolution $I$ of $M$ such that $\Ass I^i \subseteq \Phi^\complement$ for all $0 \le i \le n$.

\end{enumerate}
\end{prop}

\begin{proof}
The equivalence (2) $\Leftrightarrow$ (3) is obvious.
If (3) holds, then $\Gamma_\Phi(I^i)=0$ for all $0\le i\le n$ by Proposition \ref{1}(2), and hence $\H_\Phi^i(M)=\H^i(\Gamma_\Phi(I))=0$ for all $0\le i\le n$.
Thus (3) implies (1).

We prove the implication (1) $\Rightarrow$ (3) by induction on $n$.
When $n=0$, there are inclusions 	
$$
M \hookrightarrow \mathrm{D}_\Phi^0(M) = \mathrm{H}^0(\E(M)/ \Gamma_\Phi(\E(M))) \hookrightarrow \E^0(M)/ \Gamma_\Phi(\E^0(M)) \in \Ass^{-1}_{\Mod R} (\Phi^\complement)
$$
by Remark \ref{trans}(2) and Proposition \ref{lmd}.
Let $n\ge1$.
Then, as $\H_\Phi^0(M)=0$, the induction basis yields an exact sequence $0\to M\to I^0\to N\to0$ with $I^0\in\Ass_{\Inj R}^{-1}(\Phi^\complement)$.
The long exact sequence shows $\H_\Phi^i(N)=0$ for all $0\le i\le n-1$.
Applying the induction hypothesis to $N$, we get an injective resolution $0\to N\to I^1\to I^2\to\cdots$ with $\Ass I^i\subseteq\Phi^\complement$ for all $1\le i\le n$.
Splicing the above two exact sequences, we obtain a desired injective resolution $0\to M\to I^0\to I^1\to\cdots$.
\end{proof}

We recall the definition of a balanced big Cohen--Macaulay module over a local ring, whose existence in full generality has recently been proved by Andr\'e \cite{And}.

\begin{dfn}
Let $R$ be a local ring with maximal ideal $\m$.
An $R$-module $B$ is called a {\em balanced big Cohen--Macaulay $R$-module} if $\m B\ne B$ and every system of parameters of $R$ is a $B$-regular sequence.
\end{dfn}

\begin{rem}\label{108}
Let $R$ be a local ring.
Let $I$ be an ideal of $R$ with $\dim R/I=\dim R$.
Then every balanced big Cohen--Macaulay $R/I$-module is a balanced big Cohen--Macaulay $R$-module.
This is straightforward from the definition.
\end{rem}

We give a necessary condition for $\supp^{-1}(\Phi^\complement)$ to be $n$-wide, whose proof uses a balanced big Cohen--Macaulay module.

\begin{prop}\label{111}
Let $n\ge0$ be an integer.
Let $\Phi$ be a specialization-closed subset of $\Spec R$.
Suppose that $\supp^{-1}_{\Mod R}(\Phi^\complement)$ is $n$-wide.
Then $\H_\Phi^i(R)_\p=0$ for all integers $i>n$ and all prime ideals $\p$ of $R$ with $\height\p\le i$.
\end{prop}

\begin{proof}
In view of Remarks \ref{rem-}(3), \ref{rem2}(2) and Proposition \ref{1}(5), it suffices to prove that $\H_{\Phi_\p}^{n+1}(R_\p)=0$ for every prime ideal $\p$ with $\height\p=n+1$.
Assume contrarily that $\H_{\Phi_\p}^{n+1}(R_\p)\ne0$ for some such prime ideal $\p$.
Let $S$ be the completion of the local ring $R_\p$.
Then $\dim S=\dim R_\p=n+1$.
Applying \cite[Theorem 2.8]{HD}, we find a prime ideal $P$ of $S$ with $\dim S/P=n+1$ such that $\P S+P$ is $\p S$-primary for all $\P\in\Phi_\p$.
It follows from \cite{And} that there exists a balanced big Cohen--Macaulay $S/P$-module $B$.
Then $B$ is also a balanced big Cohen--Macaulay $S$-module by Remark \ref{108}.
We have $\mu_i(\p S,B)=0$ for all $i<n+1$ and $\mu_{n+1}(\p S,B)>0$ by \cite[Theorem 3.2]{Sha}.
Using Proposition \ref{sup0}(1)(2), we observe that
$$
\Ass_S(\E_S^i(B))\subseteq\supp_S(B)\setminus\{\p S\}\subseteq\Supp_S(B)\setminus\{\p S\}\subseteq\V(\ann_SB)\setminus\{\p S\}\subseteq\V(P)\setminus\{\p S\}\subseteq\Psi^\complement
$$
for all $i<n+1$, where $\Psi:=\{Q\in\Spec S\mid Q\cap R_\p\in\Phi_\p\}=\{Q\in\Spec S\mid Q\cap R\in\Phi\}$ .
Note that $\E_S^i(B)$ is also injective as an $R$-module.
Proposition \ref{1}(1b) shows $\Gamma_\Psi(\E_S^i(B))=0$, from which we easily see that $\Gamma_\Phi(\E_S^i(B))=0$.
Propositions \ref{sup0}(2) and \ref{1}(1b) imply $\supp_R(\E_S^i(B))=\Ass_R(\E_S^i(B))\subseteq\Phi^\complement$, whence $\E_S^i(B)\in\X:=\supp^{-1}(\Phi^\complement)$ for all $i<n+1$.

Now suppose that $\X$ is $n$-wide.
Then the image $C$ of the $n$th differential map in $\E_S(B)$ belongs to $\X$.
We have $\e_S(C)=\E_S^{n+1}(B)$, which contains $\e_S(S/\p S)$ as a direct summand.
There are injective homomorphisms $R/\p\hookrightarrow S/\p S\hookrightarrow \e_S(C)$.
It is seen that $\p\in\Ass_RC\subseteq\supp_RC\subseteq\Phi^\complement$, which implies $\Phi_\p=\emptyset$ since $\Phi$ is specialization-closed.
This contradicts our assumption that $\H_{\Phi_\p}^{n+1}(R_\p)\ne0$.
Consequently, $\X$ is not $n$-wide.
\end{proof}

Now we give several answers to Question \ref{q} (see (4) and (5) of Remark \ref{rem2}).
The second assertion recovers \cite[(2)\,$\Leftrightarrow$\,(4) in Theorem 4.9]{epi}.

\begin{thm}\label{01}
Let $\Phi$ be a specialization-closed subset $\Phi$ of $\Spec R$.
\begin{enumerate}[\rm(1)]
\item
$\cd\Phi\le 0$ if and only if $\supp^{-1}(\Phi^\complement)$ is localizing.
\item
$\cd\Phi\le 1$ if and only if $\supp^{-1}(\Phi^\complement)$ is wide.
\item
Suppose that $d:=\dim R$ is such that $0<d<\infty$.
Then $\cd\Phi\le d-1$ if and only if $\supp^{-1}(\Phi^\complement)$ is $(d-1)$-wide.
\end{enumerate}
\end{thm}

\begin{proof}
The ``only if'' parts of the three assertions follow from Theorem \ref{nser} and (4) and (5) of Remark \ref{rem2}.
So we have only to prove the ``if'' parts.
Set $\X=\supp^{-1}(\Phi^\complement)$.

(1) Fix an $R$-module $M$ and put $N=M/\Gamma_\Phi(M)$.
Proposition \ref{1} implies $\Ass N\subseteq\Phi^\complement$, and $\e_R(N)$ is in $\X$ by Proposition \ref{sup0}(2).
Now assume that $\X$ is localizing.
Then $\X$ is closed under submodules, and hence $N$ belongs to $\X$.
By Remark \ref{trans}(1) we have $\R\Gamma_\Phi(N)=0$.
Using Proposition \ref{1}(3), we obtain $\H_\Phi^{>0}(M)=0$.

(2) Fix an $R$-module $M$ and put $N=\d_\Phi^0(M)$.
Proposition \ref{trs}(4) implies $\H_\Phi^{\le1}(N)=0$, and $\Ass\E^i(N)\subseteq\Phi^\complement$ for $i=0,1$ by Proposition \ref{6}.
Proposition \ref{sup0}(2) implies $\E^i(N)\in\X$ for $i=0,1$.
Now, suppose that $\X$ is wide.
Then $\X$ is closed under kernels, and hence $N$ is in $\X$.
Remark \ref{trans}(1) shows $\R\Gamma_\Phi(N)=0$.
By Proposition \ref{trs}(5) we conclude $\H_\Phi^{\ge2}(M)=0$.

(3) Suppose that $\X$ is $(d-1)$-wide.
Then it is observed by Proposition \ref{111} that $\H_\Phi^i(R)_\p=0$ for all integers $i\ge d$ and all prime ideals $\p$.
Hence $\H_\Phi^i(R)=0$ for all integers $i\ge d$, and Proposition \ref{cdr} concludes $\cd\Phi\le d-1$.
\end{proof}

The following proposition gives a necessary condition for $n$-wideness.
For an ideal $I$ of $R$ we denote by $\d(I)$ the set of prime ideals of $R$ not containing $I$.

\begin{prop}\label{110}
Let $I$ be an ideal of $R$.
Let $n\ge0$ be an integer.
Let $M$ be a finitely generated $R$-module with $IM\ne M$.
If $\supp^{-1}\d(I)$ is $n$-wide, then $\grade(I,M)\le n$.
\end{prop}

\begin{proof}
Suppose contrarily that $\grade(I,M)>n$.
Then $\H_I^{\le n}(M)=0$, and $\Ass\E^i(M)\subseteq\d(I)$ for all $i\le n$ by Proposition \ref{6}.
Using Proposition \ref{sup0}(2), we have $\E^i(M)\in\supp^{-1}\d(I)$ for all $i\le n$.
There is an exact sequence $0\to M\to\E^0(M)\to\cdots\to\E^n(M)$, and since $\supp^{-1}\d(I)$ is closed under $n$-kernels, we get $M\in\supp^{-1}\d(I)$.
We obtain $\Ass\E^i(M)\subseteq\d(I)$ for all $i\ge0$ by Proposition \ref{sup0}(2) again, and $\H_I^i(M)=0$ for all $i\in\Z$ by Proposition \ref{6} again.
This contradicts \cite[Theorem 6.2.7]{BS}, and thus $\grade(I,M)\le n$.
\end{proof}

As an application of this proposition, we give an example of a subcategory which is precisely $n$-wide.

\begin{ex}
Let $\xx= x_1, \ldots, x_n$ be a sequence of elements of $R$.
Then $\supp^{-1}\d(\xx)$ is $n$-wide and which is not $(n-1)$-wide if $\xx$ is an $R$-regular sequence.
\end{ex}

\begin{proof}
Since $I:=(\xx)$ is generated by $n$-elements, one has $\cd I \le n$ and hence $\supp^{-1}\d(I)$ is $n$-wide by Theorem \ref{nser}.
If $\xx$ is a regular sequence, then $\grade I=n$ and the subcategory $\supp^{-1}\d(\xx)$ is not $(n-1)$-wide by Proposition \ref{110}.
\end{proof}

Recall that an ideal $I$ of $R$ is called {\em perfect} if $\pd_RR/I=\grade I$.
We now obtain the following theorem, which also gives affirmative answers to Question \ref{q}.

\begin{thm}\label{100}
\begin{enumerate}[\rm(1)]
\item
Suppose either that
\begin{enumerate}[\rm(i)]
\item
$R$ has prime characteristic $p>0$ and $I$ is a perfect ideal, or that
\item
$R$ is a semigroup ring $R= S[\Delta]$ for some affine semigroup (i.e., finitely generated commutative monoid) $\Delta$ and some noetherian ring $S$, and $I$ is a perfect ideal generated by elements of $\Delta$.
\end{enumerate}
The following are equivalent.\\
{\rm(a)} $\cd I\le n$.\quad
{\rm(b)} $\supp^{-1}\d(I)$ is $n$-wide.\quad
{\rm(c)} $\grade(I,M)\le n$ for all $M\in\mod R$ with $IM\ne M$.\quad
{\rm(d)} $\grade I\le n$.
\item
Let $(R,\m,k)$ be a complete regular local ring of Krull dimension $d$ such that $k$ is separably closed.
Let $I$ be an ideal of $R$ such that $\V(I)\setminus\{\m\}$ is connected.
Then $\cd I\le d-2$ if and only if $\supp^{-1}\d(I)$ is $(d-2)$-wide.
\end{enumerate}
\end{thm}

\begin{proof}
(1) The implication (c) $\Rightarrow$ (d) is clear, while (a) $\Rightarrow$ (b) $\Rightarrow$ (c) follow from Theorem \ref{nser} and Proposition \ref{110}.
The proofs of \cite[Lemma 2.1 and Corollaries 2.2, 2.4]{V} show (d) $\Rightarrow$ (a).

(2) The ``only if'' part follows from Theorem \ref{nser}.
Let us show the ``if'' part.
Proposition \ref{110} shows $\grade I\le d-2$, which implies $\dim R/I\ge 2$ since $R$ is a Cohen--Macaulay local ring.
We obtain $\cd I\le d-2$ by \cite[Theorem 1.1]{HL}.
\end{proof}

We state another result in relation to Question \ref{q}.

\begin{thm}\label{10}
The following are equivalent for an integer $n\ge0$.
\begin{enumerate}[\rm(1)]
\item
The subcategory $\supp^{-1}((\Max R)^\complement)$ is $n$-wide.
\item
The subcategory $\supp^{-1}\Phi$ is $n$-wide for every generalization-closed subset $\Phi$ of $\Spec R$.
\item
There is an inequality $\cd\Phi \le n$ for every specialization-closed subset $\Phi$ of $\Spec R$.
\item
The inequality $\dim R \le n$ holds.
\end{enumerate}
\end{thm}

\begin{proof}
The implications (4) $\Rightarrow$ (3) and (3) $\Rightarrow$ (2) follow from Remark \ref{rem-}(3) and Theorem \ref{nser} respectively, while (2) $\Rightarrow$ (1) is obvious.
Let us show (1) $\Rightarrow$ (4).
Assume $\dim R>n$.
Then there is a maximal ideal $\m$ of $R$ with $h:=\dim R_\m=\height\m>n$.
Since $\supp^{-1}((\Max R)^\complement)$ is $n$-wide, Proposition \ref{111} implies $\H_{\Max R}^h(R)_\m=0$.
Using Proposition \ref{1}(5), we have $0=\H_{(\Max R)_\m}^h(R_\m)=\H_{\m R_\m}^h(R_\m)$, which gives a contradiction.
We conclude that $\dim R\le n$.
\end{proof}

The following result is a direct consequence of Theorems \ref{nser} and \ref{10}.

\begin{cor}
It holds that $\cd(\Max R)<\infty$ if and only if $\dim R<\infty$.
\end{cor}

The reader may wonder if there exists a specialization-closed subset $\Phi$ of $\Spec R$ with $\cd\Phi=\infty$.
The above corollary shows that there actually does: whenever $R$ has infinite Krull dimension, $\Max R$ is such a specialization-closed subset.

Finally, we clarify that Question \ref{q} always has a negative answer for $2\le n\le\dim R-2$.
For this, we prove a proposition.

\begin{prop}\label{106}
Let $1\le r\le s$ be integers.
Let $\a,\b$ be ideals of $R$ generated by $r,s$ elements, respectively.
Put $I=\a\cap\b$ and $J=\a+\b$.
Assume $\H_J^{r+s}(R)\ne0$.
Then $\cd I=r+s-1$, and $\supp^{-1}\d(I)$ is $s$-wide.
\end{prop}

\begin{proof}
The Mayer--Vietoris sequence \cite[3.2.3]{BS} gives an exact sequence $\cdots\to\H_J^i(M)\to\H_\a^i(M)\oplus\H_\b^i(M)\to\H_I^i(M)\to\H_J^{i+1}(M)\to\cdots$ for each $R$-module $M$.
As $\a,\b$ (resp. $J$) are generated by less than $r+s$ (resp. $r+s+1$) elements, we have $\H_\a^i(M)\oplus\H_\b^i(M)=\H_J^{i+1}(M)=0$ for all $i\ge r+s$.
Hence $\H_I^{\ge r+s}(M)=0$.
The exact sequence $\H_I^{r+s-1}(R)\to\H_J^{r+s}(R)(\ne0)\to\H_\a^{r+s}(R)\oplus\H_\b^{r+s}(R)(=0)$ shows $\H_I^{r+s-1}(R)\ne0$.
It follows that $\cd I=r+s-1$.

Since $\a,\b$ are generated by $r,s$ elements, we have $\cd\a\le r$ and $\cd\b\le s$.
Theorem \ref{nser} shows that $\supp^{-1}\d(\a)$ is $r$-wide and $\supp^{-1}\d(\b)$ is $s$-wide.
It follows from Remark \ref{rem2}(2)(3) that $\supp^{-1}\d(I)=\supp^{-1}\d(\a)\cap\supp^{-1}\d(\b)$ is $s$-wide.
\end{proof}

The following example immediately follows from Proposition \ref{106}.

\begin{ex}\label{102'}
Let $\xx=x_1,\dots,x_t$ be a sequence of elements of $R$ such that $\height(\xx)=t\ge 4$.
Let $\frac{t}{2} \le s \le t-2$ be an integer.
Consider the ideal $I=(x_1,\dots,x_{t-s})\cap(x_{t-s+1},\dots,x_t)$.
Then $\supp^{-1}\d(I)$ is $s$-wide, but $\cd I=(t-s) + s -1>s$.
\end{ex}

This example says that Question \ref{q} has a negative answer whenever $2\le n\le\dim R-2$ (for example, one can take $s=n$ and $t=n+2$).

\begin{ac}
Part of this work was done during Takahashi's visits to the University of Kansas from March 2018 to September 2019, and to the Vietnam Institute for Advanced Study in Mathematics (VIASM) in the spring of 2019 invited by the working group of Nam, Tri and Yen.
He thanks them for their hospitality.
Finally, the authors thank the anonymous referee for reading the paper carefully and giving helpful comments.
\end{ac}

\end{document}